\documentclass{amsart}

\usepackage{amssymb,latexsym,amsfonts,amsmath, amsthm}
\usepackage{url}
\usepackage{subfigure}
\usepackage{algorithm2e}
\usepackage{verbatim}
\usepackage{color}
\usepackage{diagrams}
\usepackage{graphicx}

\newtheorem{theorem}{Theorem}[section]
\newtheorem{lemma}[theorem]{Lemma}

\newtheorem{problem}[theorem]{Problem}

\newtheorem{definition}[theorem]{Definition}

\newtheorem{remark}[theorem]{Remark}

\newcommand{\R}{{\mathbb{R}}}

\newcommand{\N}{{\mathbb{N}}}

\newcommand{\ie}{{\it i.e.}}

\newcommand{\Post}{\mathrm{Post}}

\begin{document}
\title[Symbolic Approximate Time-Optimal Control]{Symbolic Approximate Time-Optimal Control}
\author{Manuel Mazo Jr and Paulo Tabuada}
\thanks{This work has been partially
supported by the National Science Foundation CAREER award 0717188.}
\thanks{M. Mazo Jr is with INCAS$^3$, Assen and the Department of Discrete Technology and Production Automation, University of Groningen, The Netherlands, 
{\tt\small m.mazo@rug.nl}\\
P. Tabuada is with the Department of Electrical Engineering, University of California, Los Angeles, CA 90095-1594,{\tt\small tabuada@ee.ucla.edu}%
}

\maketitle








%

\begin{abstract}
There is an increasing demand for controller design techniques capable of
addressing the complex requirements of todays embedded applications.
This demand has sparked the interest in symbolic control where lower
complexity models of control systems are used to cater for complex
specifications given by temporal logics, regular languages, or automata. These
specification mechanisms can be regarded as qualitative since they divide the
trajectories of the plant into bad trajectories (those that need to be
avoided) and good trajectories. However, many applications
require also the optimization of quantitative measures of the trajectories
retained by the controller, as specified by a cost or utility function. As a first step towards the synthesis of controllers reconciling both
qualitative and quantitative specifications, we investigate in this paper
the use of symbolic models for time-optimal controller synthesis. We
consider systems related by approximate (alternating) simulation relations and
show how such relations enable the transfer of time-optimality information
between the systems. We then use this insight to synthesize approximately
time-optimal controllers for a control system by working with a lower
complexity symbolic model. The resulting approximately time-optimal
controllers are equipped with upper and lower bounds for the time to reach a
target, describing the quality of the controller. 
The results described in
this paper were implemented in the Matlab Toolbox
\textsf{Pessoa}~\cite{pessoaURL} which we used to workout several illustrative
examples reported in this paper.
\end{abstract}



\section{Introduction}
\label{sec:intro}

Symbolic abstractions are simpler descriptions of control systems, typically
with finitely many states, in which each symbolic state represents a collection or
aggregate of states in the control system. The power of abstractions has
been exploited in the computer science community over the years, and only
recently started to gather the attention of the control systems community.
In the present paper we analyze the suitability of symbolic
abstractions of control systems to synthesize controllers enforcing both
qualitative and quantitative specifications. 

Qualitative specifications require the
controller to preclude certain undesired trajectories from the system to be controlled.
The term qualitative refers to the fact that all the desired trajectories are
treated as being equally good. Examples of qualitative specifications
include requirements given by means of temporal-logics, $\omega$-regular
languages, or automata on infinite strings. These specifications are hard (if
not impossible) to address with classical control design theories. In
practice, most solutions to such problems are obtained through hierarchical
designs with supervisory controllers on the top layers. Such designs are
usually the result of an ad-hoc process for which correctness
guarantees are hard to obtain. Moreover, these kinds of designs require a
certain level of insight that just the most experienced system designers
posses. Recent work in symbolic control~\cite{girard09, pola08, symbolicTAC06}
has emerged as an alternative to ad-hoc designs.

In many practical applications, while there are plant
trajectories that must be eliminated, there is also
a need to select the best of the remaining trajectories. Typically, the best
trajectory is specified by means of a cost or utility associated to each
trajectory. The control design problem then requires the removal of the
undesirable trajectories and the selection of the minimum cost or maximum
utility trajectory.
As a first step towards our objective of
synthesizing controllers enforcing qualitative and quantitative
objectives, we consider in the present paper the synthesis of time-optimal
controllers for reachability specifications.
A problem of this kind, widely studied in the robotics literature, is that of optimal kinodynamic motion planning. Such problem is known to easily become computationally hard~\cite{canny_complexity_1988}. We discuss in Section~\ref{ssec:complexity} where the complexity of solving this kind of problems resides when following our methods.

Since the illustrious seminal contributions in the 50's by
Pontryagin~\cite{pontryagin59} and
Bellman~\cite{bellman52}, the design of optimal controllers has remained a
standing quest of the controls community. Despite the several advances since
then, solving optimal control problems with complex
geometries on the state space, constraints in the input space, and/or complex
dynamics is still a daunting task. This has motivated the development of
numerical techniques to solve complex optimization problems. A common method
in the literature is to discretize the dynamics and apply optimal search
algorithms on graphs such as Dijkstra's
algorithm~\cite{grune09, broucke05}. The philosophy behind such work is to show that by using finer discretizations, one obtains controllers that are arbitrarily close to the optimal controller. In contrast, our objective is not to approach the optimal solution asymptotically, but rather to \emph{effectively} compute an approximate solution and to establish how much it deviates from the optimal one.
Other techniques to solve complex optimal control problems include Mixed
(Linear or Quadratic) Integer Programing~\cite{karaman08} and
SAT-solvers~\cite{bemporad06}.

The approach we follow in the present paper is complementary to the
aforementioned techniques and our contribution is twofold:

\begin{itemize}

\item At the theoretical level, we show that time-optimality information
can be transferred from a system $S_a$ to a system $S_b$ when system $S_a$ is
related to system $S_b$ by an approximate (alternating) simulation relation.
Hence, we decouple the analysis of optimality considerations from the design of algorithms extracting a discretization $S_a$ from the original system $S_b$.
Using this result, we  show how to construct an approximately time-optimal
controller for system $S_b$ from a time-optimal controller for system $S_a$.
Moreover, we also provide bounds on how much the cost or utility of the
approximately time-optimal controller deviates from the true cost or utility.
These bounds are often conservative due to the, in general, non-deterministic
nature of the abstractions used. However, these bounds 
can still be useful in practice as performance guarantees for the obtained solutions.

\item At the practical level, we illustrate the practicality of our results by implementing them in 
the freely available Matlab toolbox
\textsf{Pessoa}~\cite{pessoaCAV,pessoaURL}. We report on several examples
conducted in \textsf{Pessoa} to illustrate the feasibility of the proposed approach.
\end{itemize}


The proposed results are independent of the specific
techniques employed in the construction of symbolic abstractions provided that
the existence of approximately (alternating) simulations relations is
established. The specific constructions reported
in~\cite{tabuada-09, zamani10journal}
show that our assumptions can be met for a large class of
systems, thus making the use of the proposed methods
widely applicable. Furthermore, effective algorithms and data structures from
computer science can be used to implement the proposed techniques, see for
example the recent work on optimal synthesis~\cite{bloem09}. In particular,
the examples presented in the current paper, performed in the Matlab
toolbox \textsf{Pessoa}, were implemented using Binary Decision Diagrams
(BDD's)~\cite{Wegener00branchingprograms} to store systems modeling both
plants and controllers. The fact that BDD's can be used to automatically
generate hardware~\cite{bloem07} or software~\cite{sangiovanni99}
implementations of the controllers makes them specially attractive.

The paper is organized as follows: in Section~\ref{sec:prelim} we review the notions of systems and  relationships between systems. Section~\ref{sec:time_optimal_control} formalizes the optimal control problem studied in this paper, and establishes relationships between the attainable costs for two systems related by (alternating) simulation relationships.  Section~\ref{sec:solution} provides an algorithm to solve time-optimal control problems approximately by relying on symbolic abstractions.
For the convenience of the readers wishing to solve concrete time-optimal
problems, we provide a concise description of all the necessary steps in Section~\ref{SSec:Practice}. Some illustrative examples are presented in Section~\ref{sec:bdd_and_example} and Section~\ref{sec:discussion} concludes the paper with a brief discussion.

\section{Preliminaries}
\label{sec:prelim}

\subsection{Notation}
\label{subsec:notation}

Let us start by introducing some notation that will be used throughout the
present paper. We denote by $\N$ the natural numbers including zero and by
$\N^+$ the strictly positive natural numbers. With $\R^+$ we denote the strictly
positive real numbers, and with $\R^+_0$ the positive real numbers including
zero. 
The identity map on a set $A$ is denoted
by $1_{A}$. If $A$ is a subset of $B$ we denote by
\mbox{$\imath_{A}:A\hookrightarrow B$} or simply by $\imath $ the natural inclusion
map taking any $a \in A$ to \mbox{$\imath (a) = a \in B$}. 
The
closed ball centered at $x\in{\mathbb{R}}^{n}$ with radius $\varepsilon$ is defined by
\mbox{$\mathbf{B}_{\varepsilon}(x)=\{y\in{\mathbb{R}}^{n}\,|\,\Vert
x-y\Vert\leq\varepsilon\}$}. We denote by $\mathbf{int}(A)$ the interior of a
set $A$. A \emph{normed vector space} $V$ is a vector space equipped with a norm $\Vert\cdot\Vert$, as is well-known this induces the metric $\mathbf{d}(x,y)=\Vert x-y\Vert,\,\, x,y\in V$.
Given a vector
\mbox{$x\in\mathbb{R}^{n}$} we denote by $x_{i}$ the $i$--th element of $x$ and by
$\Vert x\Vert$ the infinity norm of $x$; we recall that \mbox{$\Vert
x\Vert=max\{|x_1|,|x_2|,...,|x_n|\}$}, where $|x_i|$ denotes the absolute
value of $x_i$. 
We identify a relation $R\subseteq A\times B$ with the map
\mbox{$R:A\rightarrow2^{B}$} defined by \mbox{$b\in R(a)$} iff \mbox{$(a,b)\in R$}. For a set $S\in A$ the set $R(S)$ is defined as \mbox{$R(S)=\lbrace b\in B\, :\, \exists\, a\in S\,, (a,b)\in R \rbrace$}.
Also, $R^{-1}$ denotes the inverse relation
defined by \mbox{$R^{-1}=\{(b,a)\in B\times A:(a,b)\in R\}$}.
We also denote by $\mathbf{d}:X\times X\to\R^+_0$ a metric in the space $X$
and by\linebreak
\mbox{$\pi_X:X_a\times X_b \times U_a\times U_b\to X_a\times X_b$} the
projection sending\linebreak
\mbox{$(x_a,x_b,u_a,u_b)\in X_a\times X_b \times U_a\times U_b$} to
$(x_a,x_b)\in X_a\times X_b$.

\subsection{Systems}
\label{subsec:systems}

In the present paper we use the mathematical notion of  \emph{systems} to
model dynamical phenomena. This notion is formalized in the following definition:

\begin{definition}[System~\cite{tabuada-09}]
A system $S$ is a sextuple 
$(X,X_0,U,\rTo,Y,H)$ 
consisting of:

\begin{itemize}

\item a set of states $X$;
\item a set of initial states $X_0\subseteq X$
\item a set of inputs $U$;
\item a transition relation $\rTo\subseteq X\times U\times X$;
\item a set of outputs $Y$;
\item an output map $H:X\rightarrow Y$.
\end{itemize}
A system is said to be:

\begin{itemize}

\item \textit{metric}, if the output set $Y$ is equipped with a metric
$\mathbf{d}:Y\times Y\rightarrow\mathbb{R}_{0}^{+}$;
\item \textit{countable}, if $X$ is a countable set;
\item \textit{finite}, if $X$ is a finite set.
\end{itemize}
\end{definition}

We use the notation $x\rTo^uy$ to denote \mbox{$(x,u,y)\in\rTo$}.
For a transition $x\rTo^uy$, state $y$ is called a $u$-successor,
or simply successor. We denote the set of \mbox{$u$-successors} of a state $x$
by $\Post_u(x)$. If for all states $x$ and inputs $u$ the sets
$\Post_u(x)$ are singletons (or empty sets) we say the system $S$
is \emph{deterministic}. If, on the other hand, for some state $x$ and input $u$
the set $\Post_u(x)$ has cardinality greater than one, we say that system
$S$ is \emph{non-deterministic}. Furthermore, if there exists some pair $(x,u)$
such that $\Post_u(x)=\emptyset$ we say the system is
\emph{blocking}, and otherwise \emph{non-blocking}.
We also use the notation $U(x)$ to denote the set 
\mbox{$U(x)=\lbrace u\in U |\Post_u(x)\neq\emptyset\rbrace$}.

Nondeterminism arises for a variety of reasons such as modeling simplicity.
Nevertheless, to every nondeterministic system $S_a$ we can associate a
deterministic system $S_{d(a)}$ by extending the set of
inputs:
\begin{definition}[Associated deterministic system]
\label{def:Staueta_det}
The deterministic system $S_{d(a)}=(X_a,X_{a0},U_{d(a)},\rTo_{d(a)},Y_a,H_a)$ associated with a given
system \linebreak
\mbox{$S_a=(X_a,X_{a0},U_{a},\rTo_{a},Y_a,H_a)$}, is defined by:
\begin{itemize}

  \item $U_{d(a)}=U_a\times X_a$;
  \item $x\rTo^{(u,x')}_{d(a)}x'$ if there exists $x\rTo_{a}^{u}x'$ in $S_a$.
\end{itemize}
\end{definition}

Sometimes we need to refer to the possible sequences of
outputs that a system can exhibit. We call these sequences of
outputs \emph{behaviors}. Formally, behaviors are defined as follows:

\begin{definition}[Behaviors~\cite{tabuada-09}]
For a system $S$ and given any state $x\in X$, a \emph{finite
behavior} generated from $x$ is a finite sequence of transitions:
$$
y_0\rTo y_1\rTo y_2\rTo\ldots\rTo y_{n-1}\rTo y_{n}
$$
such that $y_0=H(x)$ and there exists a sequence of states $\lbrace
x_i\rbrace$, and a sequence of inputs $\lbrace u_i\rbrace$ satisfying:
$H(x_i)=y_i$ and $x_{i-1}\rTo^{u_{i-1}}x_{i}$ for all $0\leq i< n$.

An \emph{infinite behavior} generated from $x$ is an infinite sequence
of transitions: 
$$
y_0\rTo y_1\rTo y_2\rTo y_3\rTo\ldots
$$
such that $y_0=H(x)$ and there exists a sequence of states $\lbrace
x_i\rbrace$, and a sequence of inputs $\lbrace u_i\rbrace$ satisfying:
$H(x_i)=y_i$ and $x_{i-1}\rTo^{u_{i-1}}x_{i}$ for all $i\in\N$.
\end{definition}

By $\mathcal{B}_x(S)$ and $\mathcal{B}^{\omega}_x(S)$ we denote the
set of finite and infinite external behaviors generated from
$x$, respectively. Sometimes we use the notation \linebreak
\mbox{$\mathbf{y}=y_0y_1y_2\ldots y_n$}, to denote external behaviors, and $\mathbf{y}(k)$
to denote the $k$-th output of the behavior,\ie, $y_k$. A behavior $\mathbf{y}$
is said to be \emph{maximal} if there is no other behavior containing
$\mathbf{y}$ as a prefix.

Our objective is to design time-optimal controllers for control systems,
which are formalized in the following definition:

\begin{definition}[Continuous-time control system]
\label{def:ControlSys}A \textit{continuous-time control system} is a triple
$\Sigma=(\mathbb{R}^{n},\mathcal{U},f)$ consisting of:

\begin{itemize}

\item the state set $\mathbb{R}^{n}$;
\item a set of input curves $\mathcal{U}$ whose elements are essentially
bounded piece-wise continuous functions of time from intervals of the form
$]a,b[\subseteq\mathbb{R}$ to \mbox{$\mathsf{U}\subseteq \R^m$} with \mbox{$a<0<b$};
\item a smooth map $f:\mathbb{R}^{n}\times \mathsf{U}\rightarrow\mathbb{R}^{n}$.
\end{itemize}
A piecewise continuously differentiable curve
\mbox{$\xi:]a,b[\rightarrow\mathbb{R}^{n}$} is said to be a
\textit{trajectory} or \textit{solution} of $\Sigma$ if there exists
$\upsilon\in\mathcal{U}$ satisfying:
\[
\dot{\xi}(t)=f(\xi(t),\upsilon(t)),
\]
for almost all $t\in$ $]a,b[$.
\end{definition}

Although we have defined trajectories over open domains, we shall refer to
trajectories $\xi:[0,\tau]\rightarrow\mathbb{R}^{n}$ defined on closed
domains $[0,\tau],$ $\tau\in\mathbb{R}^{+}$ with the understanding of the
existence of a trajectory $\xi^{\prime}:]a,b[\rightarrow\mathbb{R}^{n}$
such that $\xi=\xi^{\prime}|_{[0,\tau]}$. We also write
$\xi_{x\upsilon}(t)$ to denote the point reached at time $t\in[0,\tau]$
under the input $\upsilon$ from initial condition $x$; this point is
uniquely determined, since the assumptions on $f$ ensure existence and
uniqueness of trajectories. 

\subsection{Systems relations}
\label{subsec:relations}

The results we prove build upon certain
simulation relations that can be established between systems. The
first relation explains how a system can simulate another system.

\begin{definition}[Approximate Simulation
Relation~\cite{tabuada-09}]\index{simulation!approximate}
\label{def:ApproxSimulationRelation}
Consider two metric systems $S_a$ and $S_b$ with $Y_a=Y_b$, and let
$\varepsilon\in\R_0^+$. A relation \mbox{$R\subseteq X_a\times X_b$} is an
\emph{$\varepsilon$-approximate simulation relation} from $S_a$ to $S_b$ if
the following three conditions are satisfied:
\begin{enumerate}
\item for every $x_{a0} \in X_{a0}$, there exists $x_{b0} \in X_{b0}$ with
$(x_{a0},x_{b0}) \in R$;
\item for every $(x_a,x_b)\in R$ we have
\mbox{$\mathbf{d}(H_a(x_a),H_b(x_b))\le\varepsilon$};
\item for every $(x_a,x_b)\in R$ we have that $x_a\rTo^{u_a}_a x_a'$
in $S_a$ implies the existence of $x_b\rTo_b^{u_b} x_b'$ in $S_b$ satisfying
$(x_a',x_b')\in R$.
\end{enumerate}
We say that $S_a$ is $\varepsilon$-approximately simulated by $S_b$ or that
$S_b$ $\varepsilon$-approximately simulates $S_a$, denoted by $S_a
\preceq_{\mathcal{S}}^\varepsilon S_b$, if there exists an
$\varepsilon$-approximate simulation relation from $S_a$ to $S_b$.
\end{definition}

When $S_a\preceq_{\mathcal{S}}^\varepsilon S_b$, system $S_b$ can replicate
the behavior of system $S_a$ by starting at a state $x_{b0}\in X_{b0}$ related
to any initial state $x_{a0}\in X_{a0}$ and by replicating every transition in
$S_a$ with a transition in $S_b$ according to (3). It then follows from (2)
that the resulting behaviors will be the same up to an error of $\varepsilon$.
If $\varepsilon=0$ the second condition implies that two states $x_a$ and
$x_b$ are related whenever their outputs are equal, \ie, $(x_a,x_b)\in R$ implies
$H(x_a)=H(x_b)$, and we say that the relation is an \emph{exact} simulation
relation. When nondeterminisn is regarded as adversarial, the notion of
approximate simulation can be modified by explicitly accounting for
nondeterminisn.

\begin{definition}[Approximate alternating simulation
relation~\cite{tabuada-09}]\index{simulation!approximate alternating}
\label{def:ApproxAlternatingSimulation}
Let $S_a$ and $S_b$ be metric systems with $Y_a=Y_b$ and let
$\varepsilon\in\R_0^+$. A relation \mbox{$R\subseteq X_a\times X_b$} is an
\emph{$\varepsilon$-approximate alternating simulation relation} from $S_a$ to
$S_b$ if the following three conditions are satisfied:
\begin{enumerate}
\item for every $x_{a0}\in X_{a0}$ there exists $x_{b0}\in X_{b0}$ with
$(x_{a0},x_{b0})\in R$;
\item for every $(x_a,x_b)\in R$ we have
\mbox{$\mathbf{d}(H_a(x_a),H_b(x_b))\le\varepsilon$};
\item for every $(x_a,x_b)\in R$ and for every \mbox{$u_a\in U_a(x_a)$} there
exists \mbox{$u_b\in U_b(x_b)$} such that for every \mbox{$x_b'\in \Post_{u_b}(x_b)$} there
exists $x_a'\in\Post_{u_a}(x_a)$ satisfying $(x_a',x_b')\in R$.
\end{enumerate}
We say that $S_a$ is $\varepsilon$-approximately alternatingly simulated by
$S_b$ or that $S_b$ \mbox{$\varepsilon$-approximately} alternatingly simulates $S_a$,
denoted by $S_a \preceq_{\mathcal{AS}}^\varepsilon S_b$, if there exists an
$\varepsilon$-approximate alternating simulation relation from $S_a$ to $S_b$.
\end{definition}

Note that for deterministic systems the notion of alternating
simulation degenerates into that of simulation. In general, the notions of
simulation and alternating simulation are incomparable as illustrated by Example 4.21
in~\cite{tabuada-09}. Also note that for any system
$S_a$, its deterministic counterpart $S_{d(a)}$ satisfies
$S_a\preceq^0_{\mathcal{AS}}S_{d(a)}$. As in the case of exact simulation
relations, we say a \mbox{0-approximate} alternating simulation relation is an
\emph{exact alternating simulation relation}.

\subsection{Composition of systems}
The feedback composition of a controller $S_c$ with a plant $S_a$ describes the
concurrent evolution of these two systems subject to synchronization
constraints. In this paper we use the notion of extended alternating simulation relation to describe these constraints. The following formal definition is only used in the proof of Lemma~\ref{lemma:costs_rel_abstract}. The readers not interested in the proof can simply replace the symbol $S_c\times^{\varepsilon}_\mathcal{F}S_a$, defined below, with ``controller $S_c$ acting on the plant $S_a$''. 

\begin{definition}[Extended alternating simulation relation~\cite{tabuada-09}]
Let $R$ be an alternating simulation relation from system $S_a$ to system
$S_b$. The extended alternating simulation relation $R^e\subseteq X_a\times X_b
\times U_a \times U_b$ associated with $R$ is defined by all the quadruples
$(x_a,x_b,u_a,u_b)\in X_a\times X_b
\times U_a \times U_b $ for which the following three conditions hold:
\begin{enumerate}
  \item $(x_a,x_b)\in R$;
  \item $u_a\in U_a(x_a)$;
  \item $u_b\in U_b(x_b)$ and for every $x_b^\prime\in\Post_{u_b}(x_b)$ there
  exists $x_a^\prime\in\Post_{u_a}(x_a)$ satisfying $(x_a^\prime,x_b^\prime)\in R$.
\end{enumerate}
\end{definition}

The interested reader is referred to~\cite{tabuada-09} for a detailed
explanation on how the following notion of feedback composition guarantees
that the behavior of the plant is restricted by controlling only its inputs.

\begin{definition}[Approximate feedback composition~\cite{tabuada-09}]
Let $S_c$ and $S_a$ 
be two metric systems with the same output sets $Y_c=Y_a$, normed vector
spaces, and let $R$ by an $\varepsilon$-approximate alternating simulation
relation from $S_c$ to $S_a$. The feedback composition of $S_c$ and $S_a$ with
interconnection relation $\mathcal{F}=R^e$, denoted by
$S_c\times^{\varepsilon}_\mathcal{F}S_a$, is the system
$(X_{\mathcal{F}},X_{\mathcal{F}},U_{\mathcal{F}},\rTo_{\mathcal{F}},Y_{\mathcal{F}},H_{\mathcal{F}})$
consisting of:
\begin{itemize}

  \item $X_\mathcal{F}=\pi_X(\mathcal{F})=R$;
  \item $X_{\mathcal{F}0}=X_\mathcal{F}\cap(X_{c0}\times X_{a0})$;
  \item $U_\mathcal{F}=U_c \times U_a$;
  \item $(x_c,x_a)\rTo^{(u_c,u_a)}_{\mathcal{F}}(x_c^\prime,x_a^\prime)$ if the
  following three conditions hold:\begin{enumerate}
                          \item $(x_c,u_c,x_c^\prime)\in\rTo_c$;
                          \item $(x_a,u_a,x_a^\prime)\in\rTo_a$;
                          \item $(x_c,x_a,u_c,u_a)\in\mathcal{F}$;
						\end{enumerate}
  \item $Y_\mathcal{F}=Y_c=Y_a$;
  \item $H_\mathcal{F}(x_c,x_a)=\frac{1}{2}(H(x_c)+H(x_a))$.
\end{itemize}
\end{definition}

We also denote by $S_c\times_\mathcal{F}S_a$ exact feedback compositions
of systems, \ie,~whenever $\mathcal{F}=R^e$ with $R$ an exact ($\varepsilon=0$)
alternating simulation relation.

\section{Time-optimal control and simulation relations}
\label{sec:time_optimal_control}

In this section we provide the main theoretical contribution of this paper
by explaining how approximate simulation relations can be used to relate
time-optimality information.

\subsection{Problem definition}
\label{subsec:problem_def}

To simplify the
presentation, we consider only systems in which $X_a=Y_a$ and $H_a=1_{X_a}$. 
However, all the results in this paper can be easily extended to systems
with $X_a\ne Y_a$ and $H_a\ne 1_{X_a}$ as we explain at the end of Section~\ref{sec:solution}.

\begin{problem}[Reachability]
Let $S_a$ be a system with \mbox{$Y_a=X_a$} and \mbox{$H_a=1_{X_a}$}, and let
$W\subseteq X_a$ be a set of outputs. Let $S_c$ be a controller and $R$ an alternating
simulation relation from $S_c$ to $S_a$.
The pair $(S_c,\mathcal{F})$, with $\mathcal{F}=R^e$, is said to solve the
reachability problem if there exists $x_{0}\in X_{\mathcal{F}0}$ such that for
every maximal behavior $\mathbf{y}\in\mathcal{B}_{x_{0}}(S_c\times_{\mathcal{F}}S_a)
\cup\mathcal{B}_{x_{0}}^{\omega}(S_c\times_{\mathcal{F}}S_a)$, there exists
$k(x_{0})\in\N$ for which $\mathbf{y}(k(x_{0}))=y_{k(x_{0})}\in W$.
\end{problem}
We denote by $\mathcal{R}(S_a, W)$ the set of
controller-interconnection pairs $(S_c,\mathcal{F})$ that solve the reachability
problem for system $S_a$ with the target set $W$ as specification. For brevity,
in what follows we refer to the pairs $(S_c,\mathcal{F})$ simply as
\emph{controller pairs}.

\begin{definition}[Entry time]
\label{Def:EntryTime}
Let $S$ be a system and let $W\subseteq X$ be a subset of outputs. The \emph{entry time} of
$S$ into $W$ from $x_0\in X_0$, denoted by
\mbox{$J(S,W,x_0)$}, is the minimum $k\in\N$ such
that for all maximal behaviors $\mathbf{y}\in\mathcal{B}_{x_0}(S)
\cup\mathcal{B}_{x_0}^{\omega}(S)$, there exists some
$k'\in[0,k]$ for which $\mathbf{y}(k')=y_{k'}\in W$.
\end{definition}
If the set $W$ is
not reachable from state $x_0$ we define
$J(S, W,x_0)=\infty$.
Note that asking in Definition~\ref{Def:EntryTime} for the \emph{minimum} $k$ is needed
because $S$ might be a non-deterministic
system, and thus there might be more than one behavior contained in 
$\mathcal{B}_{x_0}(S) 
\cup\mathcal{B}_{x_0}^{\omega}(S)$ and entering $W$.

If system $S$ is the result of the feedback composition of a system $S_a$ and a
controller $S_c$ with interconnection relation $\mathcal{F}$, \ie,
$S=S_c\times_{\mathcal{F}}S_a$, we denote by $\tilde{J}(S_c,\mathcal{F},S_a,
W,x_{a0})$ the minimum entry time over all possible initial states of the controller related to $x_{a0}$:
$$
\tilde{J}(S_c,\mathcal{F},S_a, W,x_{a0})=\min_{x_{c0}\in X_{c0}}\lbrace
J(S_c\times_{\mathcal{F}}S_a, W,(x_{c0},x_{a0}))\;\big|\;
(x_{c0},x_{a0})\in X_{\mathcal{F}0}\rbrace $$


The time-optimal control problem asks for the selection of the minimal
entry time behavior for every $x_0\in X_0$ for which $J(S,W,x_0)$ is finite.

\begin{problem}[Time-optimal reachability]
Let $S_a$ be a system with \mbox{$Y_a=X_a$} and $H_a=1_{X_a}$, and let
$ W\subseteq X_a$ be a subset of the set of outputs of $S_a$. The
time-optimal reachability problem asks to find the controller pair \linebreak
\mbox{$(S^*_c,\mathcal{F}^*)\in\mathcal{R}(S_a, W)$} such that for any other
pair $(S_c,\mathcal{F})\in\mathcal{R}(S_a, W)$ the following is
satisfied:
$$
\forall x_{a0} \in X_{a0},\;
\tilde{J}(S_c,\mathcal{F},S_a, W,x_{a0})\geq
\tilde{J}(S^*_c,\mathcal{F}^*,S_a, W,x_{a0}).
$$
\end{problem}

\subsection{Entry time bounds}
\label{subsec:cost_bounds}

The entry time $J$ acts as the cost function we aim at minimizing by designing
an appropriate controller.
The following Lemma, which is quite insightful in itself, explains how the
existence of an approximate alternating simulation relates the minimal entry
times of each system.

\begin{lemma}
\label{lemma:costs_rel_abstract}
Let $S_a$ and $S_b$ be two 
systems with $Y_a=X_a$, $H_a=1_{X_a}$,
$Y_b=X_b$ and $H_b=1_{X_b}$, and let $W_a\subseteq X_a$ and $W_b\subseteq X_b$
be subsets of states. If the following two conditions are satisfied:
\begin{itemize}

  \item \mbox{$S_a\preceq^{\varepsilon}_{\mathcal{AS}}S_b$} with the
  relation $R_\varepsilon \subseteq X_a\times X_b$;
  \item $R_\varepsilon(W_a)\subseteq W_b$
\end{itemize}
then the following holds:
$$
(x_{a0},x_{b0})\in R_\varepsilon\implies 
 \tilde{J}(S^*_{ca},\mathcal{F}^*_a,S_a, W_a,x_{a0})\geq
 \tilde{J}(S^*_{cb},\mathcal{F}^*_b,S_b, W_b,x_{b0}) $$
where $(S^*_{ca},{\mathcal{F}^*_a})\in\mathcal{R}(S_a, W_a)$ and
$(S^*_{cb},{\mathcal{F}^*_b})\in\mathcal{R}(S_b, W_b)$ denote the time-optimal
controller pairs for their respective time-optimal control problems, and
\mbox{$x_{a0}\in X_{a0}$}, $x_{b0}\in X_{b0}$.
\end{lemma}
\begin{proof}

We prove the result by parts.
In the case when\linebreak 
$ \tilde{J}(S^*_{ca},\mathcal{F}^*_a,S_a, W_a,x_{a0})=\infty$, the result is trivially true.
Thus, we analyze the case when
$\tilde{J}(S^*_{ca},\mathcal{F}^*_a,S_a, W_a,x_{a0})<\infty$. In this case,
we show that there exists a controller $S_c$ for $S_b$ such that: 
\begin{equation}
\label{eq:proof_desire}
\tilde{J}(S_c,\mathcal{G},S_b, W_b,x_{b0})\leq
\tilde{J}(S^*_{ca},\mathcal{F}^*_a,S_a, W_a,x_{a0}).
\end{equation}
This is proved by showing
that for every maximal behavior \linebreak
$\mathbf{y^b}\in\mathcal{B}_{(x_{c0},x_{b0})}(S_c\times_\mathcal{G}^{\varepsilon}
S_b) \cup\mathcal{B}_{(x_{c0},x_{b0})}^{\omega}(S_c\times_\mathcal{G}^{\varepsilon} S_b)$ there exists a maximal behavior $\mathbf{y}^a\in\mathcal{B}_{(x_{ca0},x_{a0})}(S^*_{ca}\times_{\mathcal{F}^*_a}S_a)
\cup\mathcal{B}_{(x_{ca0},x_{a0})}^{\omega}(S^*_{ca}\times_{\mathcal{F}^*_a}S_a)$ $\varepsilon$-related to $\mathbf{y}^b$.
The proof is finalized by noting that to be optimal, the controller $(S^*_{cb},\mathcal{F}^*_b)$ has to satisfy:
$$\tilde{J}(S^*_{cb},\mathcal{F}^*_b,S_b, W_b,x_{b0})\leq\tilde{J}(S_c,\mathcal{G},S_b, W_b,x_{b0})\leq
\tilde{J}(S^*_{ca},\mathcal{F}^*_a,S_a, W_a,x_{a0})
$$ for all $x_{a0}\in X_{a0}$ and \mbox{$x_{b0}\in X_{b0}$} such that
$(x_{a0},x_{b0})\in R_\varepsilon$, hence proving the result.

We start defining the controller $S_c$ for  system $S_b$.
Let $R_a$ be the alternating simulation relation defining the interconnection relation
$\mathcal{F}^*_a=R^e_a$. 
We define an interconnection relation
$\mathcal{G}=R^e_G$ that allows us to use the system
\mbox{$S_c=S^*_{ca}\times_{\mathcal{F}^*_a}S_a$} as a controller for system $S_b$. The
interconnection relation $\mathcal{G}=R^e_G$ is determined by the relation: 
$$ R_G=\lbrace((x_{ca},x_a),x_b)\in(X^*_{ca}\times X_a)\times X_b\;
\big|\;(x_{ca},x_a)\in R_a\wedge(x_a,x_b)\in R_{\varepsilon}\rbrace. 
$$
Furthermore, one can easily prove (for a detailed explanation see Proposition~11.8
in~\cite{tabuada-09}) that
\begin{equation}
\label{eq:1proofLemma}
S_c\times_\mathcal{G}^{\varepsilon}
S_b\preceq^{\frac{1}{2}\varepsilon}_{\mathcal{S}}S_c=S^*_{ca}\times_{\mathcal{F}^*_a}S_a,
\end{equation}
with the relation $R_{cb}\subseteq X_{\mathcal{G}}\times X_c$:
$$ 
R_{cb}=\lbrace ((x_c,x_b),x_c')\in X_{\mathcal{G}} \times X_{\mathcal{F}^*_a}\;
\big|\;x_c=x_c' \rbrace. $$

In order to show that for every maximal behavior\linebreak $\mathbf{y^b}\in\mathcal{B}_{(x_{c0},x_{b0})}(S_c\times_\mathcal{G}^{\varepsilon}
S_b) \cup\mathcal{B}_{(x_{c0},x_{b0})}^{\omega}(S_c\times_\mathcal{G}^{\varepsilon} S_b)$ there exists an $\varepsilon$-related
maximal behavior $\mathbf{y}^a\in\mathcal{B}_{(x_{ca0},x_{a0})}(S^*_{ca}\times_{\mathcal{F}^*_a}S_a)
\cup\mathcal{B}_{(x_{ca0},x_{a0})}^{\omega}(S^*_{ca}\times_{\mathcal{F}^*_a}S_a)$, we first make the following remark:
for any pair $(x_{a},x_{b})\in R_\varepsilon$, by the definition of alternating simulation relation, if $U_a(x_{a})\neq\emptyset$ then $U_b(x_{b})\neq\emptyset$. 
From the definition of $\mathcal{G}$ it follows that for all $((x_{ca},x_{a}),x_{b})\in X_{\mathcal{G}}$ the pair $(x_{a},x_{b})$ belongs to $R_\varepsilon$. Thus, for any pair of related states \mbox{$(x_{a},x_{b})\in R_\varepsilon$},
there exists $x_{\mathcal{G}}\in X_{\mathcal{G}}$, namely $(x_{c},x_{b})$, with $x_{c}=(x_{ca},x_{a})$, so that \mbox{$U_c(x_{c})\neq\emptyset\implies U_\mathcal{G}(x_{\mathcal{G}})\neq\emptyset$}. 
The existence of the simulation relation~(\ref{eq:1proofLemma}) implies that for every behavior $\mathbf{y}^b$ there exists an $\varepsilon$-related behavior $\mathbf{y}^a$. Any infinite behavior is a maximal behavior, and thus we already know that for every (maximal) infinite behavior $\mathbf{y}^b$ there exists an $\varepsilon$-related (maximal) infinite behavior $\mathbf{y}^a$. Moreover,
if $\mathbf{y}^b$ is a maximal finite behavior of length $l$, the set of inputs $U_\mathcal{G}(y^b_l)$ is empty. As shown before, this implies that $U_c(y^a_l)=\emptyset$, and thus $\mathbf{y}^a$ is also maximal, where $\mathbf{y}^a$ is the corresponding behavior of $S^*_{ca}\times_{\mathcal{F}^*_a}S_a$ $\varepsilon$-related to $\mathbf{y}^b$.

We now show that~(\ref{eq:proof_desire}) holds.
For any initial state $x_{a0}$ there exists an initial controller state 
$x_{ca0}\in R_a^{-1}(x_{a0})$ of $S^*_{ca}$, such that every maximal behavior 
\mbox{$\mathbf{y}^a\in\mathcal{B}_{(x_{ca0},x_{a0})}(S^*_{ca}\times_{\mathcal{F}^*_a}S_a)
\cup\mathcal{B}_{(x_{ca0},x_{a0})}^{\omega}(S^*_{ca}\times_{\mathcal{F}^*_a}S_a)$}
reaches a state \mbox{$x_a\in W_a$} \emph{in the worst case} after $\tilde{J}(S^*_{ca},\mathcal{F}^*_a,S_a, W_a,x_{a0})$ steps. We assume in what follows that the controller is initialized at that $x_{ca0}$.
Thus, as maximal behaviors of $S_c\times_\mathcal{G}^{\varepsilon}
S_b$ are related by $R_{cb}$ to maximal behaviors \linebreak  
of $S^*_{ca}\times_{\mathcal{F}^*_a}S_a$, for any \mbox{$x_{b0}\in R_\varepsilon(x_{a0})$} every maximal behavior \linebreak
\mbox{$\mathbf{y}^b\in\mathcal{B}_{(x_{c0},x_{b0})}(S_c\times_\mathcal{G}^{\varepsilon}
S_b) \cup\mathcal{B}_{(x_{c0},x_{b0})}^{\omega}(S_c\times_\mathcal{G}^{\varepsilon} S_b)$}
reaches some state $x_b\in R_{\varepsilon}(W_a)$ in at most
$\tilde{J}(S^*_{ca},\mathcal{F}^*_a,S_a, W_a,x_{a0})$ steps. 
But then, from
the second assumption, $x_b\in R_{\varepsilon}(W_a)$ implies that $x_b\in W_b$ and we
have that
$$\tilde{J}(S_c,\mathcal{G},S_b, W_b,x_{b0})\leq
\tilde{J}(S^*_{ca},\mathcal{F}^*_a,S_a, W_a,x_{a0})$$ for all $x_{a0}\in X_{a0}$ and \mbox{$x_{b0}\in X_{b0}$} such that
$(x_{a0},x_{b0})\in R_\varepsilon$. 


\end{proof}

The second assumption in Lemma~\ref{lemma:costs_rel_abstract} requires the
sets $W_a$ and $W_b$ to be related by $R$. This assumption can always be
satisfied by suitably enlarging or shrinking the target sets.

\begin{definition}
For any relation $R\subseteq X_a\times X_b$ and any set $ W\subseteq X_b$, the sets
$\lfloor W\rfloor_R$,$\lceil W\rceil_R$ are given by:
\begin{eqnarray*}
\lfloor W\rfloor_R&=&\lbrace x_a\in X_a\;\big|\;R(x_a) \subseteq W\rbrace,\\
\lceil W\rceil_R&=&\lbrace x_a\in X_a\;\big|\; R(x_a)\cap W\neq
\emptyset\rbrace.
\end{eqnarray*}
\end{definition}

The main theoretical result in the paper explains how to obtain upper and
lower bounds for the optimal entry times in a system $S_b$ by working with a
related system $S_a$.

\begin{theorem}
\label{cor:cost_bounds}
Let $S_a$ and $S_b$ be two systems with $Y_a=X_a$, $H_a=1_{X_a}$,
$Y_b=X_b$ and $H_b=1_{X_b}$.
If $S_b$ is deterministic and there exists an approximate alternating
simulation relation $R$ from
$S_a$ to $S_b$ such that $R^{-1}$ is an approximate simulation relation from
$S_b$ to $S_a$, \ie: 
$$
S_a\preceq^{\varepsilon}_{\mathcal{AS}}S_b\preceq^{\varepsilon}_{\mathcal{S}}S_a, 
$$ then the following holds for any $W\subseteq X_b$ and $(x_{a0},x_{b0})\in
R$:
\begin{small}
$$\tilde{J}({S}^*_{cd(a)},{\mathcal{F}_d},{S}_{d(a)},\lceil W\rceil_R,
x_{a0})\leq \tilde{J}(S^*_{cb},{\mathcal{F}_b},S_b, W,x_{b0})\leq
\tilde{J}(S^*_{ca},{\mathcal{F}},S_a,\lfloor W\rfloor_R,x_{a0})$$
\end{small}
\!\!where the controller pairs $(S^*_{cb},\mathcal{F}^*_b)\in\mathcal{R}(S_b, W)$, 
$(S^*_{ca},\mathcal{F}^*_a)\in\mathcal{R}(S_a,\lfloor
W\rfloor_R)$ and
$({S}^*_{cd(a)},\mathcal{F}^*_d)\in\mathcal{R}({S}_{d(a)},\lceil
W\rceil_R)$ are optimal for their respective time-optimal control problems.
\end{theorem}

\begin{proof}
Note that
$S_b\preceq^\varepsilon_\mathcal{AS}S_{d(a)}$,
by the assumed relation and both systems being deterministic. Also note that,
by definition, $R(\lfloor W\rfloor_R)\subseteq W$ and \linebreak
\mbox{$R^{-1}(W)\subseteq\lceil W\rceil_R$}. Then the proof follows from
Lemma~\ref{lemma:costs_rel_abstract}.
\end{proof}
\begin{remark}
\label{rem:theorem}
If $S_b$ is not deterministic the inequality 
$$\tilde{J}(S^*_{cb},{\mathcal{F}_b},S_b, W,x_{b0})\leq
\tilde{J}(S^*_{ca},{\mathcal{F}},S_a,\lfloor W\rfloor_R,x_{a0})$$
still holds.
\end{remark}

Theorem~\ref{cor:cost_bounds} explains how upper and lower bounds for the
entry times in $S_b$ can be computed on $S_a$, hence decoupling the optimality
considerations from the specific algorithms used to compute the abstractions.
This possibility is of great value when $S_a$ is a much simpler system than $S_b$. 
We exploit this observation in the next section where $S_b$ denotes a control system
and $S_a$ a much simpler symbolic abstraction.

\section{Approximate time-optimal control}
\label{sec:solution}

Our ultimate objective is to synthesize time-optimal controllers to be
implemented on digital platforms. The appropriate model for this analysis
consists of a time-discretization of a control system.

\begin{definition}
\label{def:Stau}
The system $S_{\tau}(\Sigma)=(X_{\tau},X_{\tau0},U_{\tau},\rTo_{\tau},Y_{\tau},H_{\tau})$
associated with a control system $\Sigma=(\R^n,\mathcal{U},f)$ and with
\mbox{$\tau\in \R^+$} consists of:
\begin{itemize}

\item $X_\tau=\R^n$;
\item $X_{\tau0}=X_\tau$;
\item $U_\tau=\left\{\upsilon\in\mathcal{U}\,\,\vert\,\, \mathrm{dom}\,
\upsilon=[0,\tau]\right\}$;
\item $x\rTo^\upsilon_{\tau} x'$ if there exist $\upsilon\in U_{\tau}$, 
and a trajectory \mbox{$\xi_{x\upsilon}:[0,\tau]\to \R^n$} of $\Sigma$
satisfying $\xi_{x\upsilon}(\tau)=x'$;
\item $Y_{\tau}=\R^n$;
\item $H_{\tau}=1_{\R^n}$.
\end{itemize}
\end{definition}

A \emph{symbolic abstraction} of a control system is a system
in which its states represent aggregates or collections of states of the original control
system. It has been shown in~\cite{girard09, pola08, zamani10journal}
that one can construct, under mild assumptions, symbolic
abstractions in the form of finite systems $S_{abs}$ satisfying \linebreak
\mbox{$S_{abs}\preceq^\varepsilon_\mathcal{AS}S_\tau(\Sigma)\preceq^\varepsilon_\mathcal{S}S_{abs}$}
with arbitrary precision $\varepsilon$. Since $S_{abs}$
is a finite system, entry times for $S_{abs}$ can be efficiently computed by
using algorithms in the spirit of dynamic programming or Dijkstra's algorithm~\cite{EWD:NumerMath59,cormen_introduction_2001}.
It then follows from Theorem~\ref{cor:cost_bounds} that these entry times
immediately provide bounds for the optimal entry time in $S_\tau(\Sigma)$.
Moreover, the process of computing the optimal entry times
for $S_{abs}$ provides us with a time-optimal controller for $S_{abs}$ that
can be refined to an approximately time-optimal controller for
$S_\tau(\Sigma)$. The refined controller is guaranteed to enforce the bounds
for the optimal entry times in $S_\tau(\Sigma)$, computed in $S_{abs}$.

\subsection{Controller design}

We now present a fixed
point algorithm solving the time-optimal reachability problem for
finite symbolic abstractions $S_{abs}$.
We start by introducing an operator that help us define
the time-optimal controller in a more concise way.

\begin{definition}
For a given system $S_{abs}$ and target set $W\subseteq X_{abs}$, the
operator $G_W:2^{X_{abs}}\to 2^{X_{abs}}$ is defined by:
\begin{footnotesize}
\begin{eqnarray*}
G_W(Z)=\lbrace x_{abs}\in X_{abs}\;|\; x_{abs}\in
W\;\vee\;\;\exists\;u_{abs} \in U_{abs}(x_{abs})\;
\text{s.t.}\;\emptyset\neq\Post_{u_{abs}}(x_{abs})\subseteq Z
\rbrace.
\end{eqnarray*}
\end{footnotesize}
\end{definition}
A set $Z$ is said to be a \emph{fixed point} of $G_W$ if $G_W(Z)=Z$. It is shown in~\cite{tabuada-09} that when $S_{abs}$ is finite, the smallest fixed point $Z$ of $G_W$ exists and can be computed in finitely many steps by iterating $G_W$, \ie, \mbox{$Z=\lim_{i\to\infty}G^i_W(\emptyset)$.} Moreover, the reachability problem admits a solution if the minimal fixed point $Z$ of $G_W$ satisfies $Z\cap X_{abs0}\neq\emptyset$. The time-optimal controller pair can then be constructed from $Z$ as follows:


\begin{definition}[Time-optimal controller pair]
For any finite system\linebreak \mbox{$S_{abs}=(X_{abs},X_{abs0},U_{abs},\rTo_{abs},X_{abs},1_{X_{abs}})$} and for any set $W_a\subseteq X_a$, the time-optimal controller pair 
\mbox{$(S^*_{cabs},\mathcal{F}^*)\in\mathcal{R}(S_{abs},W)$} is given by the system $
S^*_{cabs}=(X_{cabs},X_{cabs0},U_{abs},\rTo_{cabs},X_{cabs},1_{X_{cabs}})
$
and by the interconnection relation $\mathcal{F^*}=R_{cabs}^e$ defined by:
\begin{itemize}

\item $
R_{cabs}=\lbrace(x_{cabs},x_{abs})\in X_{cabs}\times X_{abs}\;\big|\; x_{cabs}=x_{abs}\rbrace
$
  \item $Z=\lim_{i\to\infty}G^i_W(\emptyset)$;
  \item $X_{cabs}=Z$;
  \item $X_{cabs0}=Z\cap X_{abs0}$;
  \item $x_{cabs}\rTo^{u_{abs}}_{cabs} x_{cabs}'$ if there exists a $k\in\N^+$ such that
  \mbox{$x_{cabs}\notin G^k_W(\emptyset)$} and
  $\emptyset\neq\Post_{u_{abs}}(x_{cabs})\subseteq G^k_W(\emptyset)$,
\end{itemize}
where $\Post_{u_{abs}}(x_{cabs})$ refers to the $u_{abs}$--successors in $S_{abs}$.
\end{definition}
For more details about this controller design we refer the reader to Chapter~6
of~\cite{tabuada-09}.

\subsection{Controller refinement}
\label{ssec:refinement}

The time-optimal controller pair \mbox{$(S^*_{cabs},\mathcal{F}^*)$} 
obtained in the previous section can be easily refined into a controller pair $(S_{c\tau}(\Sigma), \mathcal{F_\tau})$ for $S_\tau(\Sigma)$. Let $R_{abs\tau}$ be the $\varepsilon$-approximate alternating simulation relation from $S_{abs}$ to $S_\tau(\Sigma)$, then the refined controller $(S_{c\tau}(\Sigma), \mathcal{F_\tau})$ is given by the system \linebreak
$
S_{c\tau}=(X_{c\tau},X_{c\tau0},U_\tau,\rTo_{c\tau},X_{c\tau},1_{X_{c\tau}})
$
and by the interconnection relation $\mathcal{F_\tau}=R_{\tau}^e$ defined by: 
\begin{itemize}

  \item $R_\tau=\lbrace(x_{c\tau},x_\tau)\in X_{c\tau}\times X_\tau\,|\,x_{c\tau}=x_\tau \rbrace$;
  \item $X_{c\tau}=X_\tau$;
  \item $X_{c\tau0}=X_{\tau0}$;
  \item $x_{c\tau}\rTo^{u_{\tau}}_{c\tau} x_{c\tau}'$ if there exists $u_{abs}=u_\tau$, $x_{cabs}\in R_{abs\tau}(x_{c\tau})$ and \linebreak $x_{cabs}'\in R_{abs\tau}(x_{c\tau}')$ such that $x_{cabs}\rTo^{u_{abs}}_{cabs} x_{cabs}'$ ,
\end{itemize}
where we assumed $U_{abs}\subseteq U_\tau$.

Intuitively, the refined controller 
enables all the inputs in $U_{cabs}(x_{abs})$ at every state $x_{\tau}\in X_\tau$ of the system $S_\tau(\Sigma)$ that is related by $R_{abs\tau}$ to the state \mbox{$x_{abs}\in X_{abs}$} of the abstraction $S_{abs}$.  
It is important to notice that this controller is non-deterministic, \ie, at a state $x_\tau$ all the inputs in \linebreak  \mbox{$U_{c\tau}(x_{\tau})=\cup_{x_{abs}\in R^{-1}_{abs\tau}(x_\tau)} U_{cabs}(x_{abs})$} are available and they all enforce the cost bounds.


\subsection{Approximate time-optimal synthesis in practice}
\label{SSec:Practice}

The following is a typical sequence of steps to be followed when
applying the presented techniques in practice. 
\begin{enumerate}
  \item \textbf{Select a desired precision $\varepsilon$.} This precision is problem dependent and given by practical margins of error.
  \item \textbf{Construct a symbolic model.} Given $\varepsilon$ construct,
  using your favorite method, a symbolic model $S_{abs}$ satisfying:
  $S_{abs}\preceq^{\varepsilon}_{\mathcal{AS}}S_\tau(\Sigma)\preceq^{\varepsilon}_{\mathcal{S}}S_{abs}$. Such abstractions can be computed using \textsf{Pessoa}~\cite{pessoaURL,pessoaCAV}.
  \item \textbf{Compute the cost's lower bound.}
  	This bound is obtained as: \linebreak
	\begin{small}
  	$\qquad\tilde{J}({S}^*_{cd(abs)},{\mathcal{F}^*_d},{S}_{d(abs)},\lceil
  	W\rceil_R, x_{abs0})=\min\lbrace
  	k\in\N^+\;\big|\;x_{abs0}\in G^k_{\lceil W\rceil_R}(\emptyset) \rbrace-1$ \end{small}
  	with $G_W$ defined for
  	system $S_{d(abs)}$. This is the best lower bound one can obtain
  	since it follows from Theorem~\ref{lemma:costs_rel_abstract} that by reducing $\varepsilon$ one
  	does not obtain a better lower bound.
  \item \textbf{Compute the cost's upper bound.} 
  This bound is obtained as: \linebreak
  	$\qquad\tilde{J}(S^*_{cabs},{\mathcal{F}^*},S_{abs},\lfloor
  	W\rfloor_R,x_{abs0})=\min\lbrace k\in\N^+\;\big|\; x_{abs0}\in G^k_{\lfloor
  	W\rfloor_R}(\emptyset)\rbrace-1$ with $G_W$ defined for system
  	$S_{abs}$. The controller obtained when computing this bound, 
  	\ie~$S^*_{cabs}$, is the time-optimal controller for $S_{abs}$ and approximately
  	time-optimal for $S_\tau(\Sigma)$ after refinement.
  \item \textbf{Iterate.} If the obtained upper bound is not acceptable,
    refine the symbolic model so that the new model $S_{abs2}$
    satisfies\footnote{The constructions in~\cite{zamani10journal} satisfy this property with 
    $\varepsilon=\eta/2$, $\varepsilon^{\prime}=\eta^{\prime}/2$ and $\varepsilon^{\prime\prime}=\frac{\eta-\eta^\prime}{2}$ by selecting $\eta^{\prime}=\frac{\eta}{\rho}$ with $\rho>1$ an odd number and $\theta=\varepsilon$, $\theta^{\prime}=\varepsilon^\prime$.}:     
    $S_{abs}\preceq^{\varepsilon^{\prime\prime}}_{\mathcal{AS}}$ $S_{abs2}\preceq^{\varepsilon^\prime}_{\mathcal{AS}}S_\tau(\Sigma)$ with $\varepsilon^\prime<\varepsilon$ and $\varepsilon^{\prime\prime}<\varepsilon$.
    In virtue of Theorem~\ref{lemma:costs_rel_abstract} (and Remark~\ref{rem:theorem}) the upper bound 
    will not increase. Moreover, it is our experience that, in general, the upper bound
    will improve by using more accurate symbolic models, \ie, $\varepsilon^\prime<\varepsilon$.
\end{enumerate}

The more general case where $X_\tau\neq Y_\tau$, $H_\tau\neq1_{X_\tau}$ and one is given an output target set $W_Y\subseteq Y$ can be solved in the same manner by using the target set $W\subseteq X$ defined by $W=H^{-1}(W_Y)$.

\subsection{Generalizations and Complexity}
\label{ssec:complexity}

We briefly discuss in this section some simple generalizations of the proposed methods and the corresponding complexity. We first note that time-optimal synthesis can be combined with safety (qualitative) objectives when the specification is given as the requirement to satisfy both a safety constraint and a reachability requirement. A controller for such specifications can be obtained by first synthesizing the least restrictive controller enforcing the safety constraint and then solving a time-optimal reachability problem. In particular, this approach can be used for specifications given as a Linear Time Logic (LTL) formula of the kind $\phi \wedge \Diamond\, p$, where $p$ is an atomic proposition denoting a set of states and $\phi$ is a formula in the safe-LTL fragment of LTL~\cite{KupfermanV01}.

The general solution of a problem including qualitative and quantitative (time-optimal) specifications consists of five steps: abstraction of the control system; translation of the safe-LTL formula into a deterministic automaton recognizing all the behaviors satisfying the formula; composition of this automaton with the finite abstraction; synthesis of a controller by solving a {\it safety} game in the finite system resulting from the composition; and finally, the synthesis of the final controller as a solution to a time-optimal {\it reachability} game in the abstraction composed with the intermediate (safety) controller.


According to the five steps solution, the (time) complexity of solving these general problems can be split in terms of those steps. The abstraction problem, following the techniques in~\cite{zamani10journal, tabuada-09} can be easily shown to have exponential complexity on the dimension of the control system; the translation of a safe-LTL formula into a deterministic automaton has doubly exponential complexity on the length of the formula~\cite{KupfermanL06}; composition of finite automata is a polynomial problem on the number of states of the composed automata; and, finally, the solution of reachability or safety games on finite automata also takes polynomial time in the number of states. This last step can be shown to be polynomial by noting that both problems admit a solution as the fixed-point of an operator~\cite{Zielonka98, tabuada-09} that needs to be iterated at most as many times as the number of states of the finite automaton.
This brief analysis indicates that the bottleneck, in general, lies on the abstraction process, as the translation of safe-LTL formulas, even though theoretically more complex, tends to be an easier problem due to the short length of the formulas used in practice.

\section{Examples}
\label{sec:bdd_and_example}

To illustrate the provided results and its practical relevance we implemented the
time-optimal controller design algorithm in Section~\ref{sec:solution} in the
publicly available Matlab toolbox named
\textsf{Pessoa}~\cite{pessoaURL, pessoaCAV}. All the run-time values for the examples where obtained on a MacBook with 2.2 GHz Intel Core 2 Duo processor and 4GB of RAM.
The abstractions generated by \textsf{Pessoa} and used in the following examples, are obtained
by discretizing the dynamics with sample time $\tau$ and the state and input sets with discretization steps $\eta$ and $\mu$ respectively. We refer the readers to~\cite{zamani10journal} where these abstractions are studied in detail. The precision $\varepsilon$ of such abstractions can be adjusted by reducing the discretization parameters $\eta$ and $\mu$.

\subsection{Double integrator}
\label{subsec:example1}

We illustrate the proposed technique on the
classical example of the double integrator, 
where $\Sigma$ is the control system: 
$$
\dot{\xi}(t)=\left[\begin{array}{cc}
                             0&1\\0&0
                             \end{array}\right]\xi(t)+\left[\begin{array}{c}
                                               0\\1
                                               \end{array}\right]\upsilon(t)                                        
$$
and the target set $W$ is the
origin, \ie, $W=\lbrace (0,0)\rbrace$. 

Following the steps presented in Section~\ref{sec:solution},
first we select a precision $\varepsilon$ which in this example we take as $\varepsilon=0.15$. Next,
we relax the problem by enlarging the target set to $W=\mathbf{B}_{1}((0,0))$.
We select as parameters for the symbolic abstraction $\tau=1$, $\mu=0.1$ and
$\eta=0.3$. Restricting the state set to
$X=\mathbf{B}_{30}((0,0))\subset\R^2$ the state set of $S_\tau(\Sigma)$ becomes finite
and the proposed algorithms can be applied. Constructing the abstraction $S_\tau(\Sigma)$ in
\textsf{Pessoa} took less than $5$ minutes
and the resulting model required $7.9$ MB to be stored. The lower bound
required about $50$ milliseconds while computing the time-optimal controller
required only $3$ seconds and the controller was stored in $1$ MB.
 
The approximately time-optimal controller $S^*_c$ is depicted in Figure~\ref{fig:symb_controller}. We remind the reader that the obtained controller is non-deterministic. Hence, Figure~\ref{fig:symb_controller} shows \emph{one} of the valid inputs of the time-optimal controller at different locations of the state-space. The optimal controller to the origin is also shown in Figure~\ref{fig:symb_controller} represented by the switching curve (thick blue line) dividing the state space into regions where the inputs $u=1$ (below the switching curve) and $u=-1$ (above the switching curve) are to be used.
As expected, the partition produced by this switching curve does not coincide with the one found by our toolbox, as the time-optimal controller reported in~\cite{pontryagin-62} is not
time-optimal to reach the set $W$ (it is just optimal when the target set is the singleton $\lbrace
(0,0)\rbrace$).

Although the computed bounds are conservative, the cost
achieved with the symbolic controller is quite close to the true optimal cost as illustrated in Figure~\ref{fig:relativetimes} and Table~\ref{tab:table1}.
This is a consequence of the bounds relying entirely on the worst case scenarios induced by the non-determinism of the computed abstractions. In practice, the symbolic controller determines the actual state of the system every time it acquires a state measurement thus resolving the nondeterminism present in the abstraction. In Figure~\ref{fig:relativetimes} we present the ratio between the cost to reach $W$, obtained from the symbolic controller, and the time-optimal controller. The time-optimal controller to reach the origin operates in continuous time and thus for some regions of the state-space the cost obtained will be smaller than one unit of time. On the other hand, the approximate time-optimal controller obtained with our techniques cannot obtain costs smaller than one unit of time, as it operates in discrete time.
Hence, to make the comparison fair, in Figure~\ref{fig:relativetimes} the  costs achieved by the time-optimal controller smaller than one unit of time were saturated to a cost of $1$ time unit.
In Table~\ref{tab:table1} specific values of the time to
reach the target set $W$ using the constructed controller are compared to the
cost of reaching $W$ with the true time-optimal controller to reach the origin.


\begin{figure}[ht]
\centering
\subfigure[]{
\includegraphics[width=0.46\hsize]{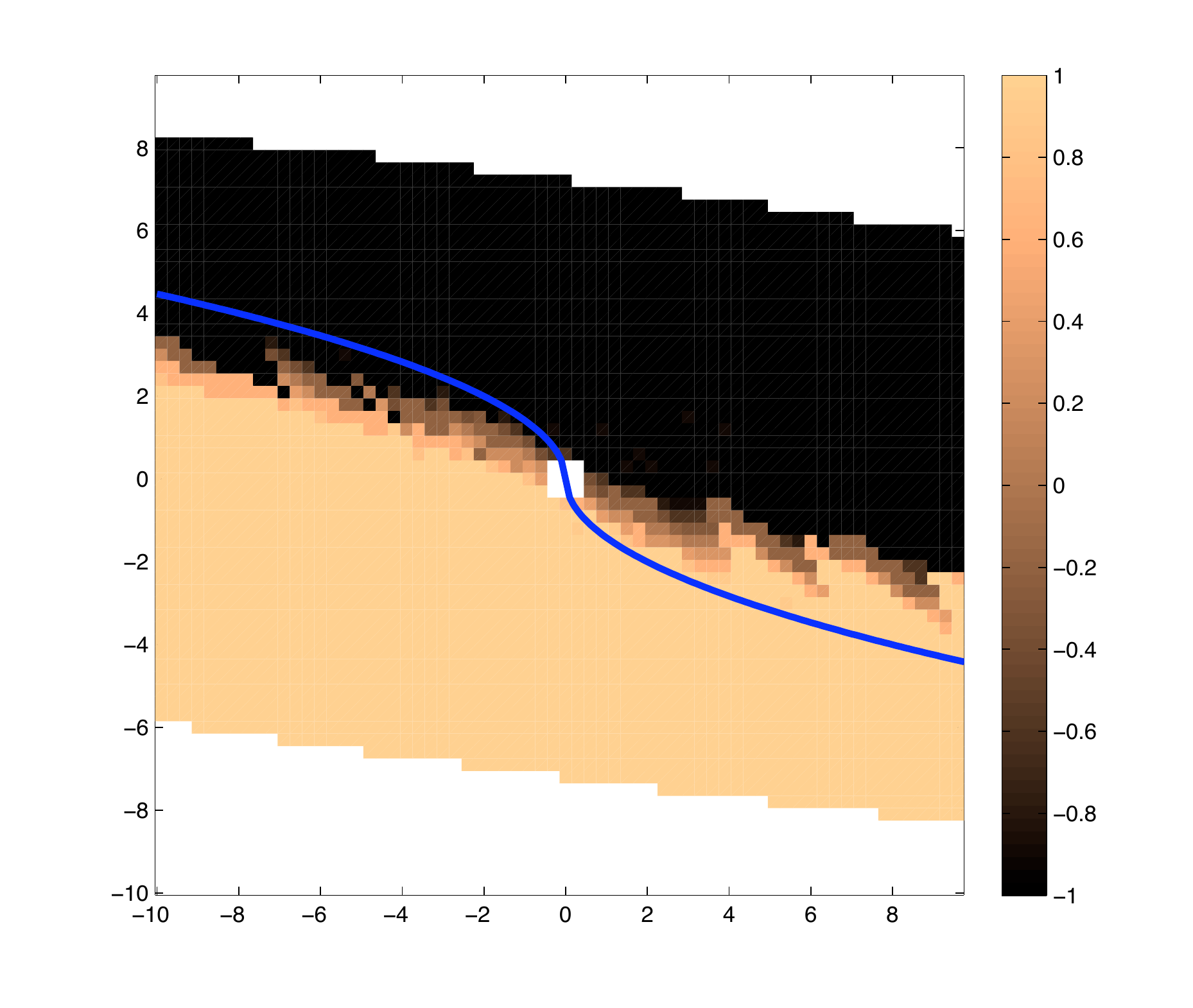}
\label{fig:symb_controller}
}
\subfigure[]{
\includegraphics[width=0.48\hsize]{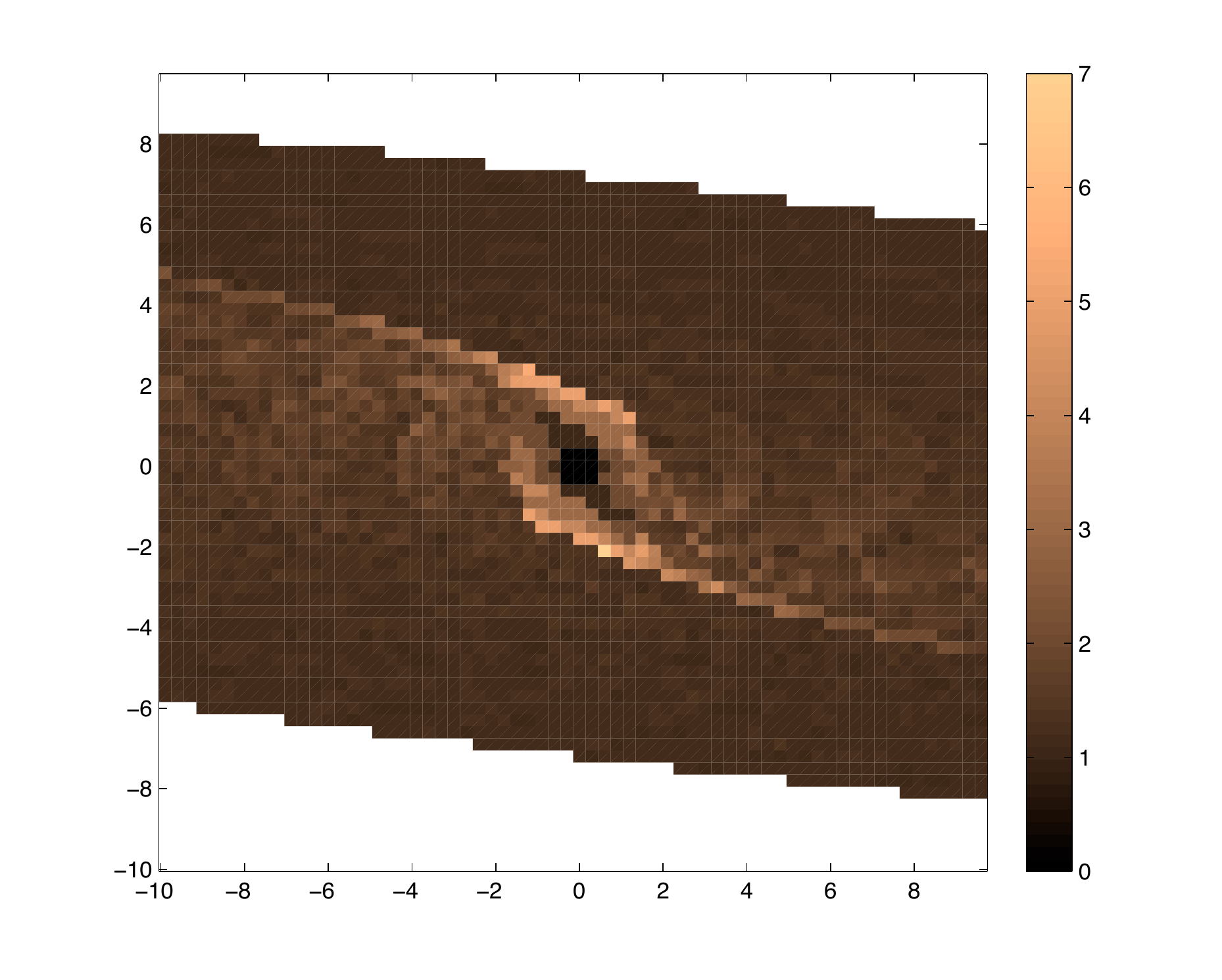}
\label{fig:relativetimes}
}
\caption{(a) Symbolic controller $S^*_c$. (b) Time to reach the target set $W$ represented as the ratio between the times obtained from the symbolic controller and the times obtained from the continuous time-optimal controller to reach the origin. }
\label{fig:fig1}
\end{figure}




\begin{table}
\begin{footnotesize}
\begin{center}
\begin{tabular}{ l c c c c c c c c}
Initial State & $(-6.1,6.1)$ & $(-6,6)$ & $(-5.85,5.85)$ & $(3.1,0.1)$ & $(3,0)$ & $(2.85,-0.1)$\\
\hline
\hline
 $ Continuous$	& $12.83$ s & $12.66$ s & $11.60$ s & $2.66$ s & $2.53$ s & $2.38$ s\\                  
 $ Symbolic $	& $14$ s & $14$ s & $13$ s & $3$ s & $3$ s & $3$ s\\
 $ Upper Bound$	& $29$ s & $29$ s & $29$ s & $7$ s & $7$ s & $7$ s\\
 $ Lower Bound$	& $9$ s & $9$ s & $9$ s & $2$ s & $2$ s & $2$ s
\end{tabular}
\end{center}
\end{footnotesize}
\caption{Times achieved in simulations by a
time-optimal controller to reach the origin and the symbolic controller. \label{tab:table1}}
\end{table}

\subsection{Unicycle example}
\label{subsec:example2}
With this example we want to persuade the reader of the potential of
the presented techniques to solve control problems with both qualitative and
quantitative specifications. The problem we consider now is to
drive a unicycle through a given environment with obstacles. In this example
both qualitative and quantitative specifications are provided. The avoidance
of obstacles prescribes conditions that the trajectories should respect, thus
establishing qualitative requirements of the desired trajectories.
Simultaneously, a time-optimal control problem is specified by requiring
the target set to be reached in minimum time, thus defining the quantitative requirements.
Hence, the complete specification requires the synthesis of a controller
disabling trajectories that hit the obstacles, and selecting, among the
remaining trajectories, those with the minimum time-cost associated to them.

We consider the following model for the unicycle control system:
\begin{equation*}
\dot{x}=v cos(\theta),\;\dot{y}=v sin(\theta),\;\dot{\theta}=\omega
\end{equation*}
in which $(x,y)$ denotes the position coordinates of the vehicle, $\theta$
denotes its orientation, and $(v,\omega)$ are the control inputs, linear
velocity and angular velocity respectively.
The parameters used in the construction of the symbolic model are:
\mbox{$\eta=0.2$, $\mu=0.1$, $\tau=0.5$ seconds}, and $v\in[
0,\,0.5]$ and $\omega\in[ -0.5,\,0.5]$. The
problem to be solved is to find a feedback controller optimally navigating the
unicycle from any initial position to the target set
$W=[4.6,5]\times[1,1.6]\times[-\pi,\pi]$, indicated with a red box in
Figure~\ref{fig:trajectory} (with any orientation $\theta$), while avoiding the
obstacles in the environment, indicated as blue boxes in Figure~\ref{fig:trajectory}.
The symbolic model was constructed in $179$ seconds and used $11.5$ MB of storage, and the approximately
time-optimal controller was obtained in $5$ seconds and required $3.5$ MB of storage. In
Figure~\ref{fig:trajectory} we present the result of applying the approximately
time-optimal controller with the prescribed qualitative requirements (obstacle
avoidance). The (approximately) bang-bang nature of the obtained controller can be appreciated in the right plot of this figure. For the initial condition $(1.5,1,0)$ the solution obtained,
presented in Figure~\ref{fig:trajectory}, required $44$ seconds to reach the
target set. 

\begin{figure}
\begin{center}
\includegraphics[width=0.36\hsize]{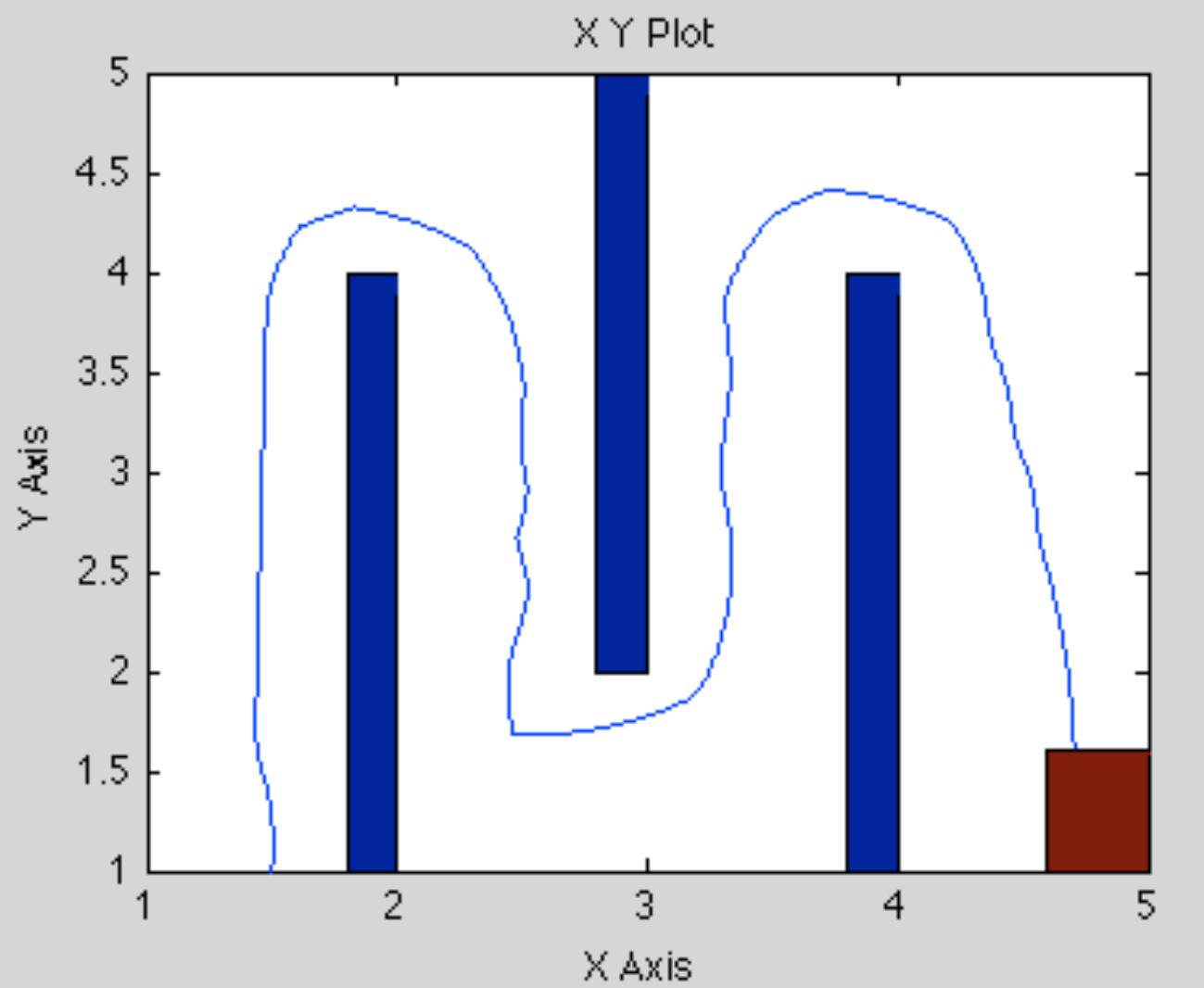}
\includegraphics[width=0.39\hsize]{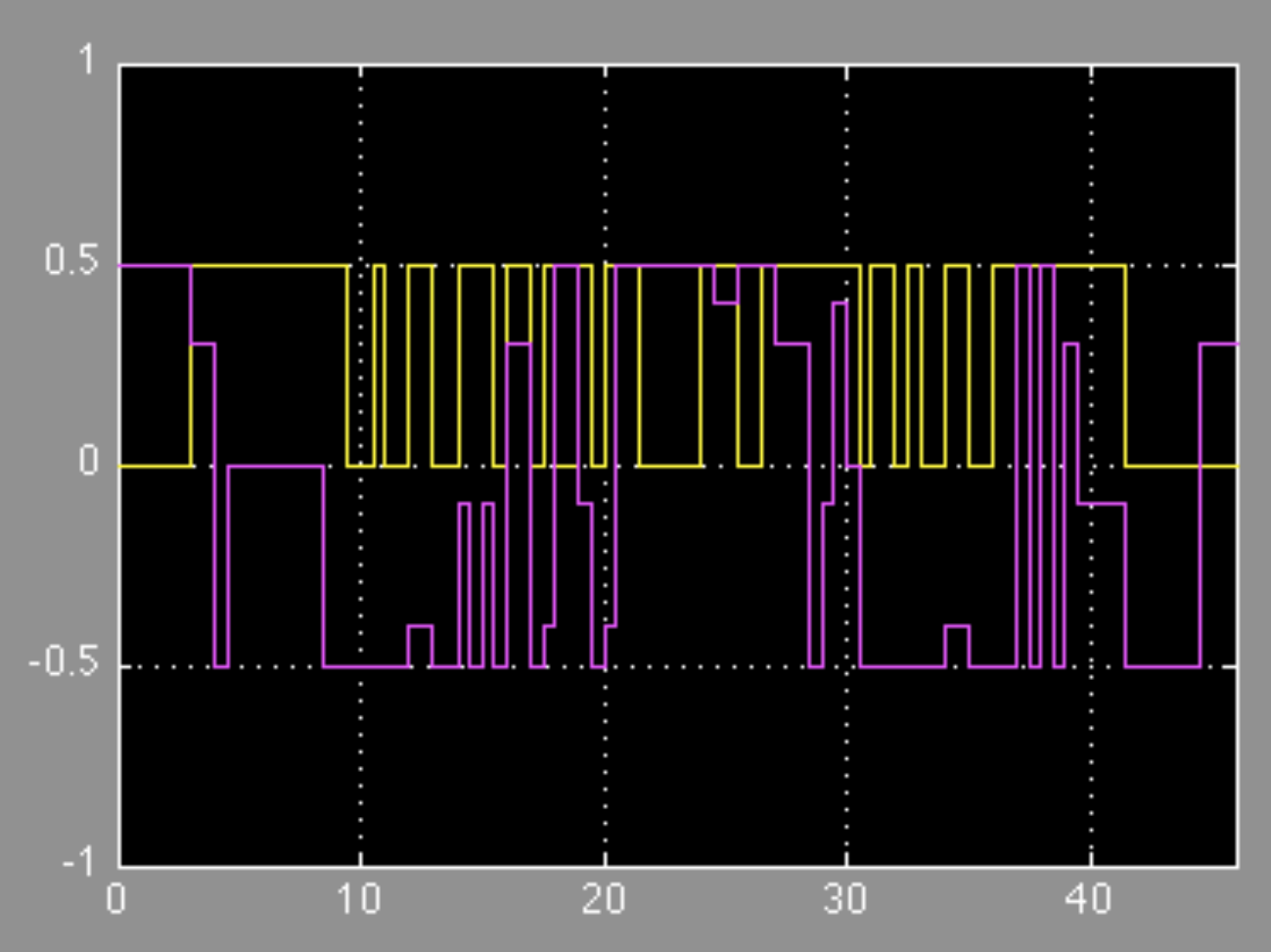}
   \caption{Unicycle trajectory under the automatically generated
   approximately time-optimal feedback controller (left figure) and the inputs employed: $v$ in yellow and $\omega$ in pink (right figure).}
   \label{fig:trajectory}
\end{center}
\end{figure}

\section{Discussion}
\label{sec:discussion}

We have proposed a computational approach to solve time-optimal control
problems by resorting to symbolic abstractions. The obtained solutions provide explicit lower
and upper bounds on the achievable cost. The employed techniques allow us to
solve complex time-optimal control problems, with target sets, state sets and
dynamics of very general nature.

The main theoretical result shows that symbolic abstractions which approximately alternatingly simulate a control system provide bounds for the achievable cost of time-optimal control problems. An 
algorithm has been provided to obtain these cost bounds by solving corresponding optimal control problems over the symbolic abstraction. Furthermore, this algorithm produces an approximately time-optimal symbolic controller that can be easily refined into a controller for the original system, as shown in Section~\ref{ssec:refinement}.
On the practical side, we have implemented the presented algorithms
in the \textsf{Pessoa} toolbox resorting to binary decision diagrams as the underlying
data structures. We have also illustrated the techniques using \textsf{Pessoa} on two
examples, the last of which illustrates how symbolic models can be used to
solve problems with both qualitative and quantitative requirements.

Future work will concentrate in the development of synthesis algorithms for combinations
of general qualitative and quantitative specifications for control systems.

\section{Acknowledgements}

The authors would like to thank Giordano Pola for the fruitful discussions in
the beginning of this project. We also acknowledge
Anna Davitian for her help in the development of \textsf{Pessoa}.

\bibliographystyle{elsarticle-num}
\bibliography{journaloptimal}

\end{document}